\documentclass[11pt]{amsart}
\author{Paul Pollack}
\address{University of Georgia\\Department of Mathematics\\Athens, Georgia 30602\\USA}
\email{pollack@uga.edu}

\author{Carlo Sanna}
\address{Universit\`a degli Studi di Torino\\Department of Mathematics\\Turin, Italy}
\email{carlo.sanna.dev@gmail.com}


\title{Uncertainty principles connected with the M\"{o}bius inversion formula}
\usepackage{amsmath,amssymb,amsthm,geometry,lgreek}
\DeclareMathAlphabet{\curly}{U}{rsfs}{m}{n}
\newtheorem{thm}{Theorem}
\newtheorem{prop}[thm]{Proposition}

\newtheorem{lem}[thm]{Lemma}
\newtheorem*{question}{Question}

\theoremstyle{remark}
\newtheorem*{rmk}{Remark}
\begin{document}
\renewcommand{\labelenumi}{(\roman{enumi})}
\def\A{\curly{A}}
\def\sumprime{\sideset{}{^{'}}{\sum}}
\def\lcm{\mathrm{lcm}}
\def\M{\curly{M}}
\def\N{\mathbf{N}}
\def\Kk{\curly{K}}
\def\V{\curly{V}}
\def\Uu{\curly{U}}
\def\1{\mathbf{1}}
\def\Q{\mathbf{Q}}
\def\Qq{\curly{Q}}
\def\Z{\mathbf{Z}}
\def\R{\mathbf{R}}
\def\Li{\mathrm{Li}}
\def\pp{\mathfrak{p}}
\def\Nm{\mathrm{N}}
\def\Tt{\curly{T}}
\def\Ss{\curly{S}}
\def\bad{\curly{P}}
\def\good{\curly{Q}}
\def\Pp{\curly{P}}
\def\Gal{\mathrm{Gal}}
\def\e{\mathrm{e}}
\def\ds{\mathbf{d}}
\def\supp{\mathrm{supp}}
\maketitle
\begin{abstract} We say that two arithmetic functions $f$ and $g$ form a \emph{M\"{o}bius pair} if $f(n) = \sum_{d \mid n} g(d)$ for all natural numbers $n$. In that case, $g$ can be expressed in terms of $f$ by the familiar M\"{o}bius inversion formula of elementary number theory. In a previous paper, the first-named author showed that if the members $f$ and $g$ of a M\"{o}bius pair are both finitely supported, then both functions vanish identically. Here we prove two significantly stronger versions of this uncertainty principle. A corollary is that in a nonzero M\"{o}bius pair, one cannot have both $\sum_{f(n) \neq 0}\frac{1}{n} <\infty$ and $\sum_{g(n) \neq 0}\frac{1}{n} <\infty$.
\end{abstract}

\section{Introduction}

Let $f$ be an arithmetic function, i.e., a function from the set of natural numbers (positive integers) to the complex numbers. The Dirichlet transform $\hat{f}$ and the M\"obius transform $\check{f}$ of $f$ are defined by the equations
\[
\hat{f}(n) := \sum_{d \mid n} f(d)\quad \text{and}\quad \check{f}(n) := \sum_{d \mid n} \mu(n/d) f(d).
\]
In a first course in number theory, one learns (M\"{o}bius inversion) that the Dirichlet and M\"{o}bius transforms are inverses of each other: That is, \[ f = \check{\hat{f}} = \hat{\check{f}} \]
for every $f$. In a short note \cite{pollack11}, the first author gave a simple proof of the following \emph{uncertainty principle for the M\"{o}bius transform}: If $f$ is an arithmetic function not identically zero, then the support of $f$ and the support of $\check{f}$ cannot both be finite. (Here the \emph{support} of an arithmetic function $h$ refers to the set $\{n: h(n)\neq 0\}$.) In this note, we present two substantial quantitative strengthenings of this result.

Call a pair of functions $(f,g)$ a M\"{o}bius pair if $f = \hat{g}$ (equivalently, if $g = \check{f}$). For the sake of typography, we state our results in terms of $f$ and $g$ rather than $f$ and $\check{f}$.

\begin{thm}\label{thm:main1} Suppose that $(f,g)$ is a nonzero M\"{o}bius pair. If \[ \sum_{n \in \supp(g)}\frac{1}{n} < \infty,\] then $\supp(f)$ possesses a positive asymptotic density. The same result holds with the roles of $f$ and $g$ reversed.
\end{thm}

Our second result is similar, but now weighted by the absolute values of $f$ and $g$. Recall that the mean value of an arithmetic function $h$ is the limit, as $x\to\infty$, of the finite averages $\frac{1}{x}\sum_{n \leq x}h(n)$.

\begin{thm}\label{thm:main2} Suppose that $(f,g)$ is a nonzero M\"{o}bius pair. If \[ \sum_{n =1}^{\infty}\frac{|g(n)|}{n} < \infty,\] then $|f|$ possesses a nonzero mean value. The same result holds with the roles of $f$ and $g$ reversed.
\end{thm}

\begin{rmk} Our Theorem \ref{thm:main2} may be compared with a classical theorem of Wintner (see \cite[p. 20]{wintner43}), according to which $\sum_{n=1}^{\infty} \frac{|g(n)|}{n} < \infty$ implies that $f$ has the (possibly vanishing) mean value $\sum_{n=1}^{\infty} \frac{g(n)}{n}$.
\end{rmk}

For the rest of this paper, call a set $\A$ of natural numbers \emph{thin} if $\sum_{a \in \A}\frac{1}{a} < \infty$. One can conclude from Theorem \ref{thm:main1} and partial summation (cf. the proof of Lemma \ref{lem:kronecker} below) that in a nonzero M\"{o}bius pair $(f,g)$,  $\supp(f)$ and $\supp(g)$ cannot both be thin. One might wonder why we bother with thin sets instead of dealing directly with natural density. The answer is given in our final theorem, which shows that the asymptotic densities of $\supp(f)$ and $\supp(g)$ can be arbitrarily prescribed. Our notation for the density of a set $\A\subset\N$ is $\ds(\A)$.

\begin{thm}\label{thm:prescribed} For any $\alpha, \beta \in [0,1]$, one can find a nonzero M\"{o}bius pair $(f,g)$ for which $\ds(\supp(f)) = \alpha$ and $\ds(\supp(g)) = \beta$. Moreover, $f$ and $g$ can be chosen as multiplicative functions.
\end{thm}

\subsection*{Notation and conventions} Throughout, the letter $p$ is reserved for a prime variable. We continue to use $\ds(\A)$ for the natural density of $\A$, defined as the limit as $x\to\infty$ of $\frac{1}{x}\#\{n\in \A: n \leq x\}$. The lower density $\underline{\ds}(\A)$ and upper density $\overline{\ds}(\A)$ are defined analogously, with $\liminf$ and $\limsup$ replacing $\lim$. We use $O$ and $o$-notation with its standard meaning. We write $p \parallel n$ when $p\mid n$ and $p^2\nmid n$. Whenever we refer to arithmetic functions $f$ and $g$, it is to be assumed that $(f,g)$ is a M\"{o}bius pair.

\section{Preliminaries}
In this section, we collect a number of lemmas needed in later arguments. The first is a well-known sufficient condition for asymptotic density to be countably additive.

\begin{lem}\label{lem:countable} Let $\A_1, \A_2, \A_3, \dots$ be a sequence of disjoint sets of natural numbers, each of which possesses an asymptotic density. Suppose that $\overline{\ds}\left(\cup_{i>k}^{\infty} \A_i\right) \to 0$ as $k\to\infty$.
Then the union of the $\A_i$ has a natural density, and in fact
\[ \ds(\cup_{i=1}^{\infty} \A_i) = \sum_{i=1}^{\infty} \ds(\A_i). \]
\end{lem}

\begin{proof} The proof is easy and so we include it here. Writing $\cup_{i=1}^{\infty} \A_i = (\cup_{i=1}^{k} \A_i) \bigcup (\cup_{i > k} \A_i)$, we find that for each $k$,
\[ \sum_{i=1}^{k} \ds{(\A_i)} \leq \underline{\ds}(\A) \leq \overline{\ds}{(\A)} \leq \left(\sum_{i=1}^{k} \ds{(\A_i)}\right) + \overline{\ds}(\cup_{i>k}\A_i).\]
Now letting $k\to\infty$ gives the lemma.
\end{proof}

If $\A$ is any set of natural numbers, we define its \emph{set of multiples} $\M(\A)$ as the collection of positive integers possessing at least one divisor from $\A$. In other words, $\M(\A):= \{an: a \in \A, n \in \N\}$. The next lemma collects two basic results on natural densities of sets of multiples.

\begin{lem}\label{lem:basicdensity} If $\A$ is thin, then $\M(\A)$ possesses an asymptotic density. Moreover, if $1 \not\in \A$, then $\ds(\M(\A)) < 1$.
\end{lem}
\begin{proof}[Proof (sketch)] The existence of $\ds(\M(\A))$ for thin sets $\A$ is due to Erd\H{o}s \cite{erdos34}. It follows from an inequality of Heilbronn \cite{heilbronn37} and Rohrbach \cite{rohrbach37} that whenever $\ds(\M(\A))$ exists,
\[ \ds(\M(\A)) \leq 1- \prod_{a \in \A}\left(1-\frac{1}{a}\right); \]
the second assertion of the lemma is now immediate. For more context for these results, one can consult \cite[Chapter 0]{hall96}.\end{proof}

According to the first half of Lemma \ref{lem:basicdensity}, the set of $n$ possessing at least one divisor from a thin set $\A$ has an asymptotic density. The next lemma allows us to draw a similar conclusion when we prescribe exactly which members $d$ of $\A$ are divisors of $n$. For technical reasons which will emerge later, we formulate this result so as to allow us to also prescribe which of the quotients $n/d$ are squarefree.

\begin{lem}\label{thin_subset_lemma}
Let $\A$ be a thin set of positive integers. If $\Tt \subset \Ss \subset \A$, where $\Ss$ is finite, then the set of positive integers $n$ for which both
	\begin{enumerate}
	\item $\Ss = \{d \in \A: d \mid n\}$, and
	\item $\Tt = {\{{d \in \Ss} : d \mid n,\; \mu(n/d) \neq 0\}}$,
	\end{enumerate}
	has an asymptotic density.
\end{lem}

\begin{proof} For each positive integer $n$ satisfying condition (i) above, define $\chi(n)$ by the equation
\begin{equation}\label{eq:indicator}
\chi(n) := \prod_{d \in \Tt} |\mu(n/d)| \prod_{c \in \Ss \setminus \Tt}(1-|\mu(n/c)|);
\end{equation}
for values of $n$ not satisfying condition (i), put $\chi(n)=0$. Then $\chi$ is the indicator function of those $n$ satisfying both (i) and (ii). Moreover, when $n$ satisfies (i), expanding the final product in \eqref{eq:indicator} reveals that
\[ \chi(n) = \sum_{\Tt \subset \Uu \subset \Ss} (-1)^{|\Uu|-|\Tt|} \prod_{e \in \Uu} |\mu(n/e)|. \] So using $'$ to denote a sum restricted to integers $n$ satisfying (i), we find that the count of $n\leq x$ satisfying both (i) and (ii) is given by
\begin{align*}
\sum_{n \leq x} \chi(n) &=  \sum_{\Tt \subset \Uu \subset \Ss} (-1)^{|\Uu|-|\Tt|} \sumprime_{n \leq x} \prod_{e \in \Uu} |\mu(n/e)|,\\
&= \sum_{\Tt \subset \Uu \subset \Ss} (-1)^{|\Uu|-|\Tt|} \#\{n \leq x: \text{$n$ satisfies (i), $n/e$ is squarefree for all $e \in \Uu$}\}.
\end{align*}
Dividing by $x$ and letting $x\to\infty$, it suffices to prove that for each set $\Uu$ with $\Tt\subset \Uu \subset \Ss$, the set
\begin{equation}\label{eq:Vdef} \V := \{n \in \mathbb{N}: \text{$n$ satisfies (i), $n/e$ is squarefree for all $e \in \Uu$}\} \end{equation}
possesses an asymptotic density.

We prove this last claim by showing that belonging to $\V$ amounts to \emph{not} lying in the set of multiples of an appropriately constructed thin set. First, notice that in order for $n$ to satisfy (i), it is necessary that $L \mid n$, where $L:= \lcm[d \in \Ss]$. So each of our candidate values of $n$ can be written in the form $n = Lq$. For condition (i) to hold, one also needs that if $a \in \A\setminus \Ss$, then $a \nmid Lq$; in other words, $a/\gcd(a,L) \nmid q$. The second condition in the definition \eqref{eq:Vdef} of $\V$ requires that for each $e \in \Uu$, the number $Lq/e$ is not divisible by any of $2^2, 3^2, 4^2, \dots$.  Equivalently, $q$  cannot be divisible by any of the numbers $h^2/\gcd(h^2, L/e)$ with $h\geq 2$. So for $n=Lq$ to belong to the set $\V$, it is necessary and sufficient that $q$ \emph{not} belong to the set of multiples of
\[ \{a/\gcd(a,L): a \in \A\setminus \Ss\} \cup \left(\bigcup_{e \in \Uu} \{h^2/\gcd(h^2, L/e): h \geq 2\}\right).\]
Call this set $\Kk$. Then $\Kk$ is a finite union of thin sets and so is thin. (We use here that $\A$ is thin, by hypothesis, and that the sum of the reciprocals of the squares converges.) Hence, $\M(\Kk)$ has a natural density, and
\[
\ds(\V) = \frac{\ds(\N \setminus \M(\Kk))}{L} = \frac{1 - \ds(\M(\Kk))}{L}.
\]
This completes the proof.
\end{proof}

We also need a simple mean-value theorem of Kronecker \cite{kronecker87}.
\begin{lem}\label{lem:kronecker} Let $h(n)$ be an arithmetic function. If $\sum_{n = 1}^{\infty} h(n)/n$ converges, then $h$ has mean value zero.
\end{lem}
\begin{proof} Let $H(x) := \sum_{n \leq x}h(n)$, and let $G(x) := \sum_{n \leq x} h(n)/n$. By hypothesis, there is a constant $\rho$ so that $G(x) = \rho + o(1)$, as $x\to\infty$. Then as $x\to\infty$,
	\begin{align*} H(x) &= x \cdot G(x) - \int_{1}^{x} G(t)\, dt \\
		& = (\rho x + o(x)) - \int_{1}^{x}(\rho + o(1))\,dt = o(x).
	\end{align*}
Hence, $h$ has mean value zero.
\end{proof}

Our final lemma, needed in the proof of Theorem \ref{thm:main2}, is a generalization of the well-known result that the squarefree numbers have asymptotic density $6/\pi^2$. The proof consists of easy sieving; compare with \cite[pp. 633--635]{landau53}.

\begin{lem}\label{lem:landau} For each natural number $P$, the set of squarefree integers relatively prime to $P$ has asymptotic density
\[ \frac{6}{\pi^2} \prod_{p \mid P} \left(1+\frac{1}{p}\right)^{-1}. \]
\end{lem}

\section{Proof of Theorem \ref{thm:main1}}

We split the proof into two parts.

\begin{proof}[Proof that if $g$ has thin support, then $\supp(f)$ has positive density] We begin by showing that the support of $f$ has \emph{some} asymptotic density. We defer showing that this density is positive to the very end of the argument.

Call two elements of $\supp(f)$ equivalent if they share the same set of divisors from $\supp(g)$. Enumerate the equivalence classes as $\A_1, \A_2, \A_3, \dots$. Then each $\A_i$ is the set of natural numbers possessing a prescribed set $\Ss_i$ of divisors from $\supp(g)$. Each $\Ss_i$ is thin (as a subset of the thin set $\supp(g)$), and so Lemma \ref{thin_subset_lemma} shows that each of the sets $\A_i$ possesses a natural density. (Note that the full strength of Lemma \ref{thin_subset_lemma} is not required, since we are only prescribing $\Ss$, not $\Tt$.) Now
\[ \supp(f) = \bigcup_{i} \A_i, \]
and the right-hand union is disjoint. If there are only finitely many classes $\A_i$, then the existence of $\ds(\supp(f))$ follows immediately from finite additivity. So suppose that there are infinitely many $\A_i$. Since only finitely many natural numbers lie below any given bound, it is clear that with
\[ m_k:= \min_{i > k} \left(\max_{d \in \Ss_i} d\right), \quad\text{we have}\quad m_k \to\infty \quad \text{as $k\to\infty$}. \]
Now if $n \in \cup_{i > k}\A_i$, then $\{d \mid n: d \in \supp(g)\} = \Ss_i$ for some $i > k$. Thus, $n$ has a divisor $d \in \supp(g)$ with $d \geq m_k$, and so
\[ \overline{\ds}(\cup_{i > k}\A_i) \leq \sum_{\substack{d \geq m_k \\ d \in \supp(g)}}\frac{1}{d}. \]
As $k\to\infty$, the right-hand side of this inequality tends to zero, because of the thinness of $\supp(g)$. So Lemma \ref{lem:countable} shows that $\ds(\supp(f))=\sum_{i} \ds(\A_i)$. This completes the proof that $\supp(f)$ has an asymptotic density.

We now show that $\ds(\supp(f))>0$. Let $d$ be the smallest member of $\supp(g)$. We claim that a positive proportion of numbers $n$ have $d$ as their only divisor from $\supp(g)$, so that $f(n) = g(d) \neq 0$. To prove the claim, note that if $n=dm$ has a divisor from $\supp(g)$ other than $d$, then $m$ is a multiple of an element from the set $\{e/\gcd(e,d): e\neq d, e \in \supp(g)\}$. That set is thin (since $\supp(g)$ is) and does not contain $1$ (by the minimality of $d$). From Lemma \ref{lem:basicdensity}, we see that the corresponding values of $m$ comprise a set of density $< 1$. Consequently, a positive proportion of multiples $n$ of $d$ have $d$ as their only divisor from $\supp(g)$, as was to be shown.
\end{proof}

\begin{proof}[Proof that if $f$ has thin support, then $\supp(g)$ has positive density] This is similar to the first half. By M\"{o}bius inversion, $g(n) = \sum_{d \mid n} f(d)\mu(n/d)$. Motivated by this, we call two elements $n_1$ and $n_2$ of $\supp(g)$ equivalent if they  share the same set of divisors $d$ from $\supp(f)$ \emph{and} $n_1/d$ and $n_2/d$ are squarefree for the same subset of these $d$. The existence of the density of $\supp(g)$ then follows from an argument entirely analogous to that seen above (but now using the full strength of Lemma \ref{thin_subset_lemma}).
	
Let $d$ be the smallest element of $\supp(f)$. To show that $\ds(\supp(g))>0$, it is enough to show that a positive proportion of $n$ have $d$ as their only divisor from $\supp(f)$ and have $n/d$ squarefree. Writing $n=dm$, we observe that these conditions hold unless $m$ belongs to the set of multiples of $\{e/\gcd(e,d): e\neq d, e \in \supp(f)\} \cup \{h^2: h\geq 2\}$. Lemma \ref{lem:basicdensity} shows that the set of these excluded values of $m$ has density $< 1$, thus completing the proof.
\end{proof}

Theorem \ref{thm:main1} is sharp in the following sense:

\begin{prop}\label{prop:best} Let $Z(x)$ be any increasing function on $[2,\infty)$ that tends to infinity (no matter how slowly). There is a nonzero M\"{o}bius pair $(f,g)$ for which
\[ \sum_{\substack{n \leq x \\ n \in \supp(g)}} \frac{1}{n} < Z(x) \]
for all large $x$, but $\supp(f)$ has asymptotic density zero.\end{prop}

\begin{proof} Let $\Pp$ be the set of primes constructed greedily by the following process: Start with $\Pp=\emptyset$. Running through the sequence of primes $2, 3, 5, \dots$ consecutively, throw $q$ into $\Pp$ if $\prod_{p \in \Pp\cap[2,q]}(1+\frac{1}{p-1}) < Z(q)$. Since $Z\to\infty$, the greedy nature of the construction guarantees that $\prod_{p \in \Pp}(1+1/(p-1))$ diverges to infinity, so that $\sum_{p \in \Pp} p^{-1}$ diverges. Now let $g$ be the completely multiplicative function with $g(p)=-1$ if $p\in \Pp$, and $g(p)=0$ otherwise. Let $f(n):=\sum_{d \mid n} g(d)$. If $p \parallel n$ where $p \in \Pp$, then $f(n)= f(p) f(n/p) = (1+g(p)) f(n/p) =0$. So if $n \in \supp(f)$, then there is no prime $p \in \Pp$ for which $p \parallel n$. Since $\Pp$ has divergent reciprocal sum, the set of $n$ satisfying this last condition has density
\[ \prod_{p \in \Pp}\left(1-\frac{1}{p} + \frac{1}{p^2}\right) = 0, \] by an elementary sieve argument (cf. \cite[Lemma 8.13, p. 260]{pollack09}). So $\supp(f)$ has density $0$. On the other hand, if $g(n) \neq 0$, then every prime factor of $n$ belongs to $\Pp$. So for large $x$,
\[ \sum_{\substack{n\leq x\\ n \in \supp(g)}} \frac{1}{n} \leq \prod_{\substack{p \leq x \\ p \in \Pp}}\left(1+\frac{1}{p}+\frac{1}{p^2} + \dots\right) =
\prod_{\substack{p \leq x \\ p \in \Pp}}\left(1+\frac{1}{p-1}\right) < Z(x).\]
This completes the proof of the proposition.\end{proof}
It is not hard to prove the analogue of Proposition \ref{prop:best} with the roles of $f$ and $g$ reversed; we leave this to the reader.

\section{Proof of Theorem \ref{thm:main2}}

We first suppose that $\sum |g(n)|/n < \infty$ and show that $|f|$ possesses a nonzero mean value. At the end of this section, we discuss the changes necessary to reverse the roles of $f$ and $g$. For each $y \geq 1$, we define an arithmetic function $f_y$ by
\[ f_y(n):= \sum_{\substack{d \mid n \\ d \leq y}} g(d); \]
analogously, we define $g_y$ by
\[ g_y(n):= \sum_{\substack{d \mid n \\ d \leq y}} \mu(n/d) f(d). \]

\begin{lem}\label{lem:lambday} Assume that $\sum_{n=1}^{\infty} \frac{|g(n)|}{n} < \infty$. Then:
\begin{enumerate}
\item For every $y$, the function $|f_y|$ possesses a finite mean value, say $\lambda_y$.
\item The constants $\lambda_y$ tend to a finite limit $\lambda$ as $y\to\infty$.
\item The mean value of $|f|$ is $\lambda$.
\end{enumerate}
\end{lem}

\begin{proof} Let $\A$ be the (finite, so also thin) set of natural numbers not exceeding $y$.  Since $f_y(n)$ depends only on the set of elements of $\A$ dividing $n$, we can write
\[ \sum_{n \leq x} |f_y(n)| = \sum_{\Ss \subset \A} \left|\sum_{d \in \Ss} g(d)\right| \sum_{\substack{n \leq x \\ \Ss = \{d\mid n:~d \in \A\}}} 1. \]
Dividing by $x$ and letting $x\to\infty$, we obtain the existence of the mean values $\lambda_y$ from Lemma \ref{thin_subset_lemma}. (We need only the half of that lemma concerned with prescribing $\Ss$.) This completes the proof of (i). To see that the $\lambda_y$ converge to a finite limit $\lambda$, notice that if $y_0 < y_1$,
\begin{align}\notag \left|\sum_{n \leq x} |f_{y_1}(n)| - \sum_{n \leq x} |f_{y_0}(n)|\right| &\leq \sum_{n \leq x} \left|f_{y_1}(n)-f_{y_0}(n)\right| \\
&\leq \sum_{n \leq x} \sum_{\substack{y_0 < d \leq y_1 \\ d \mid n}} |g(d)| \notag \\
&\leq x\sum_{\substack{d > y_0}} \frac{|g(d)|}{d}.\notag\end{align}
Dividing by $x$ and letting $x\to\infty$ shows that
\[ |\lambda_{y_1}-\lambda_{y_0}| \leq \sum_{d > y_0} \frac{|g(d)|}{d}. \]
The right-hand side tends to zero as $y_0\to\infty$. Thus, $\{\lambda_y\}$ is a Cauchy net of real numbers, and hence convergent. So we have (ii). The same arguments used to prove (ii) show that $\left|\sum_{n \leq x}|f(n)| - \sum_{n \leq x}|f_{y_0}(n)|\right| \leq x \sum_{d > y_0} |g(d)|/d$. Thus,
\[ \lambda_{y_0} - \sum_{d > y_0} \frac{|g(d)|}{d} \leq \liminf_{x\to\infty} \left(\frac{1}{x}\sum_{n\leq x}|f(n)|\right) \leq \limsup_{x\to\infty} \left(\frac{1}{x}\sum_{n\leq x}|f(n)|\right) \leq \lambda_{y_0} + \sum_{d > y_0} \frac{|g(d)|}{d}. \] Now letting $y_0\to\infty$ gives (iii).
\end{proof}

\begin{proof}[Proof that $|f|$ has a nonzero mean value, assuming $\sum_{n=1}^{\infty}\frac{|g(n)|}{n}< \infty$] It remains only to show that the number $\lambda$ from Lemma \ref{lem:lambday} is positive. Let $d$ be the smallest element of the support of $g$. We consider the partial sums of $|f|$ restricted to $n$ of the form $dm$, where every prime dividing $m$ exceeds a large but fixed real parameter $y$. Using $'$ to denote this restriction, we find that
\begin{align}
\notag\sumprime_{n \leq x} |f(n)| &\geq \sumprime_{n \leq x} \left(|g(d)| - \sum_{\substack{e \mid n \\ e \neq d}}|g(e)|\right)\\ &= \left(|g(d)|\sumprime_{n \leq x} 1\right) - \sumprime_{n \leq x}\sum_{\substack{e \mid n \\ e\neq d}} |g(e)|. \label{eq:est0} \end{align}
We proceed to estimate the remaining sums. If we put $P:=\prod_{p \leq y}p$, then
\begin{align}\notag \sumprime_{n \leq x}1 &= \sum_{\substack{m \leq x/d \\ \gcd(m,P)=1}} 1 \\&= \frac{x}{d}\prod_{p\leq y}\left(1-\frac{1}{p}\right) + O(2^{\pi(y)}), \label{eq:est1}\end{align}
where the last step follows by inclusion-exclusion (e.g., see \cite[Theorem 3.1, p. 76]{MV07}). Observe that if $n=dm$, where all of the prime factors of $m$ exceed $y$, then every divisor $e$ of $n$ belongs to $[1,d]$ or has a prime divisor $> y$. So if $e \neq d$, then the choice of $d$ forces $g(e)=0$ or $e > y$. Hence (applying inclusion-exclusion once again),
\begin{align}\notag  \sumprime_{n \leq x}\sum_{\substack{e \mid n \\ e\neq d}} |g(e)| &\leq \sum_{e > y}|g(e)| \sumprime_{\substack{n \leq x\\ e\mid n}}1 \leq \sum_{e > y}|g(e)| \sum_{\substack{m \leq x/\lcm[d,e] \\ p \mid m \Rightarrow p > y}}1 \\&\leq x \prod_{p \leq y} \left(1-\frac{1}{p}\right)\left(\sum_{e > y} \frac{|g(e)|}{\lcm[d,e]}\right) + O\left(2^{\pi(y)} \sum_{y<e \leq x} |g(e)|\right). \label{eq:est2}
\end{align}
Since $\sum_{n \geq 1} \frac{|g(n)|}{n} < \infty$, Lemma \ref{lem:kronecker} guarantees that the final error term in \eqref{eq:est2} is $o(x)$, as $x\to\infty$. Now we substitute \eqref{eq:est2} and \eqref{eq:est1} back into \eqref{eq:est0}, divide by $x$, and let $x\to\infty$ to find that
\begin{align}\notag \liminf_{x\to\infty} \frac{1}{x} \sumprime_{n \leq x} |f(n)| &\geq  \left(\frac{|g(d)|}{d} - \sum_{e > y} \frac{|g(e)|}{\lcm[d,e]}\right) \prod_{p \leq y} \left(1-\frac{1}{p}\right) \\ &\geq
\left(\frac{|g(d)|}{d} - \sum_{e > y} \frac{|g(e)|}{e}\right) \prod_{p \leq y} \left(1-\frac{1}{p}\right).\label{eq:tocompare}\end{align}
But if $y$ is fixed sufficiently large, then this last expression is positive. Since $\sum_{n \leq x}|f(n)| \geq \sum'_{n\leq x}|f(n)|$, it must be that the mean value $\lambda$ of $|f|$ is positive.
\end{proof}

We now consider the effect of swapping $f$ and $g$. That the mean value of $|g|$ exists if $\sum_{n=1}^{\infty}\frac{|f(n)|}{n}< \infty$ follows by the same arguments used to prove Lemma \ref{lem:lambday}, after swapping $f$ and $g$. There is only one substantial change necessary: The value of $f_y(n)$ depends both on the set of divisors $d$ of $n$ not exceeding $y$ \emph{and} on the subset of those $d$ for which $n/d$ is squarefree. So we must appeal to the full force of Lemma \ref{thin_subset_lemma}. We leave the remaining details to the reader.

Finally, we show that the mean value of $|g|$ is positive.

\begin{proof}[Proof that $|g|$ has a nonzero mean value, assuming $\sum_{n=1}^{\infty}\frac{|f(n)|}{n}< \infty$] We let $d$ be the least element of the support of $f$. We let $y$ denote a large but fixed real parameter, and we use $'$ with the same meaning as before. The reasoning that led us to \eqref{eq:est0} now shows that
\begin{equation}\label{eq:est01} \sumprime_{n\leq x}|g(n)| \geq |f(d)| \sumprime_{n \leq x} |\mu(n/d)| - \sumprime_{n \leq x} \sum_{\substack{e \mid n \\ e \neq d}} |f(e)|. \end{equation}
Following our earlier proof, we obtain our previous upper bound \eqref{eq:est2} for the double sum here, except now with $g$ replaced by $f$. Appealing again to Lemma \ref{lem:kronecker} and using that $\lcm[d,e]\geq e$, we thus see that
\begin{equation}\label{eq:est02} \sumprime_{n \leq x} \sum_{\substack{e \mid n \\ e \neq d}} |f(e)| \leq x \prod_{p \leq y}\left(1-\frac{1}{p}\right) \sum_{e>y}\frac{|f(e)|}{e}  + o(x), \end{equation}
as $x\to\infty$. On the other hand, one has as $x\to\infty$ that
\begin{align}\notag \sumprime_{n \leq x}|\mu(n/d)| &= \sum_{\substack{m \leq x/d \\ m \text{ squarefree}\\ p \mid m \Rightarrow p > y}} 1 \\
&\sim \frac{6}{\pi^2} \cdot \frac{x}{d} \prod_{p \leq y} \left(1 + \frac{1}{p}\right)^{-1},\label{eq:est03}\end{align}
by Lemma \ref{lem:landau} with $P:= \prod_{p \leq y}p$. Substituting \eqref{eq:est03} and \eqref{eq:est02} back into \eqref{eq:est01}, dividing by $x$, and letting $x\to\infty$, we find that
\begin{align*} \liminf_{x\to\infty} \frac{1}{x}\sumprime_{n \leq x}|g(n)| &\geq \frac{6}{\pi^2} \cdot \frac{|f(d)|}{d} \prod_{p \leq y} \left(1 + \frac{1}{p}\right)^{-1} - \prod_{p \leq y}\left(1-\frac{1}{p}\right) \sum_{e> y} \frac{|f(e)|}{e}\\
&\geq \left(\prod_{p \leq y}\left(1-\frac{1}{p}\right)\right) \left(\frac{6}{\pi^2} \cdot \frac{|f(d)|}{d} - \sum_{e> y} \frac{|f(e)|}{e}\right).
\end{align*}
This final expression is positive if $y$ is fixed sufficiently large. So the mean value of $|g|$ must be positive.
\end{proof}

\begin{rmk} Suppose that $(f,g)$ is a M\"{o}bius pair with $\sum_{n \geq 1}|g(n)|/n < \infty$. We showed above that the mean value of $|f|$ must be positive. That proof in fact shows that if every element in $\supp(g)$ is at least $d$, and $y\geq 1$, then
\[ \liminf_{x\to\infty}\frac{\sumprime_{n \leq x}|f(n)|}{\sumprime_{n \leq x}1} \geq |g(d)|-d\sum_{e > y}\frac{|g(e)|}{e}. \]
Here, as before, $\sum'$ denotes a sum restricted to integers $n$ of the form $dm$, where each prime factor of $m$ exceeds $y$. (Compare the asserted inequality with \eqref{eq:tocompare}.) A completely analogous argument shows that under the same hypotheses,
\begin{equation}\label{eq:limsupupper} \limsup_{x\to\infty}\frac{\sumprime_{n \leq x}|f(n)|}{\sumprime_{n \leq x}1} \leq |g(d)|+d\sum_{e > y}\frac{|g(e)|}{e}. \end{equation}
Here is an amusing application of \eqref{eq:limsupupper}, in the spirit of \cite{pollack11}. Suppose for the sake of contradiction that the sum of the reciprocals of the prime numbers converges. Take $g$ to be the characteristic function of the primes, take $d=1$, and take $y$ large enough that $\sum_{p > y}\frac{1}{p} < 1$. Since $f(n) = \sum_{p \mid n}1$ counts the number of distinct primes dividing $n$, we have $f(n) \geq 1$ for all $n > 1$, and so the left-hand side of \eqref{eq:limsupupper} is at least $1$. But the right hand side is smaller than $1$, a contradiction!
\end{rmk}

\section{Proof of Theorem \ref{thm:prescribed}}
\begin{proof} We first deal with the case when $\beta < 1$. Fix disjoint sets of primes $\Pp_0$ and $\Qq_0$ with $\prod_{p \in \Pp_0} (1-1/p) =\prod_{p \in \Qq_0} (1-1/p) = 0$. We can find a subset $\Pp \subset \Pp_0$ so that
\[ \prod_{p \in \Pp} (1 - 1/p) = \alpha, \quad\text{where the minimum of $\Pp$ is so large that}\quad \prod_{p \in \Pp} (1 - 1/p^2) \geq \beta.\]
We can then select a set $\Qq \subset \Qq_0$ so that $\prod_{p \in \Qq} (1 - 1/p) = \beta \prod_{p \in \Pp} (1 - 1/p^2)^{-1}$.
Define $g$ as the multiplicative arithmetic function whose values at prime powers are given by
\[
g(p^e) := \begin{cases}
-1 & \mbox{ if } p \in \Pp \mbox{ and } e = 1, \\
\phantom{+}0 & \mbox{ if } p \in \Pp \mbox{ and } e > 1, \\
\phantom{+}0 & \mbox{ if } p \in \Qq,\\
+1 & \mbox{ if } p \notin \Pp \cup \Qq.
\end{cases}
\]
Then the function $f(n):=\sum_{d\mid n}g(d)$ is multiplicative, and after computing $f$ at prime powers, we see that $\supp(f)$ consists of all the positive integers without prime factors from $\Pp$. Now an easy sieve argument (compare with \cite[Corollary 6.3, p. 173]{pollack09}) shows that
\[
\ds(\supp(f)) = \prod_{p \in \Pp}\left(1-\frac1{p}\right) = \alpha.
\]
Moreover, the support of $g$ consists of those positive integers without prime factors in $\Qq$, and which
are not divisible by any $p^2$ with $p \in \Pp$. Another simple sieve argument now shows that
\[
\ds(\supp(g)) = \prod_{p \in \Qq}\left(1-\frac1{p}\right)\prod_{p \in \Pp}\left(1-\frac1{p^2}\right) = \beta.
\]
This completes the proof when $\beta < 1$.

Consider now the case in which $\beta = 1$. This time, we select $\Pp$ as any set of prime numbers with
$\prod_{p \in \Pp}(1-1/p+1/p^2) = \alpha$; this is possible since
$\prod_{p} (1-1/p+1/p^2) = 0$. Define $g$ as the multiplicative arithmetic function satisfying
\[
g(p^e) := \begin{cases}
-1 & \mbox{ if } p \in \Pp \mbox{ and } e = 1, \\
+1 & \mbox{ if } p \in \Pp \mbox{ and } e > 1, \\
+1 & \mbox{ if } p \notin \Pp.
\end{cases}
\]
Let $f(n):= \sum_{d \mid n} g(d)$. Then $f$ is multiplicative, and computing its values at prime powers, we find that $f(n) = 0$ if and only if $p \parallel n$ for some $p \in \Pp$. Again, a sieve argument gives
\[
\ds(\supp(f)) = \prod_{p \in \Pp}\left(1-\frac1{p}+\frac1{p^2}\right) = \alpha .
\]
Since $g$ is always nonzero, we also have $\ds(\supp(g)) = 1 = \beta$.\end{proof}

\section{Concluding remarks}
Call a class of subsets of $\N$ \emph{exclusive} if for every nonzero M\"{o}bius pair $(f,g)$, at least one of the sets $\supp(f)$ and $\supp(g)$ falls outside of the class. As remarked in the introduction, our Theorem \ref{thm:main1} shows that the thin sets comprise an exclusive class.

We can do a little better than this. Call $\A \subset \N$ \emph{evaporating} if as $T\to\infty$, the upper density of the set of $n$ with a divisor in $\A \cap [T,\infty)$ tends to zero. It is easy to prove that every thin set is evaporating. However, there are evaporating sets that are not thin; an interesting example, due to Erd\H{o}s and Wagstaff \cite[Theorem 2]{EW80}, is the set of shifted primes $\{p-1\}$. We can show that the evaporating sets form an exclusive class. In fact, our Theorem \ref{thm:main1} remains true with the word ``thin'' replaced by ``evaporating''. The proof requires relatively minor modifications; the most significant of these is the replacement of Lemma \ref{lem:basicdensity} with its analogue for evaporating sets (see, e.g., \cite[Lemma 2]{PS88}).

It would be interesting to know other classes of sets that are also exclusive. The following question seems attractive and difficult.
\begin{question} For each $\delta > 0$, consider the class $\mathfrak{C}(\delta)$ of subsets $\A \subset \N$ for which
\[ \#\{n \leq x: n \in \A\} \leq C_\A \frac{x}{(\log{x})^{\delta}} \quad \text{for some real $C_\A$, and all $x\geq 2$}. \]
Is this class exclusive?
\end{question}

From Theorem \ref{thm:main1} and partial summation, $\mathfrak{C}(\delta)$ is exclusive for $\delta > 1$. Theorem \ref{thm:prescribed} shows that $\supp(f)$ and $\supp(g)$ can both have density zero; taking $\Pp = \Pp_0 = \{p \equiv 1\pmod{3}\}$ and $\Qq=\Qq_0 = \{p\equiv -1\pmod{3}\}$ in the proof of that theorem, one can show that $\mathfrak{C}(\delta)$ is not exclusive for $\delta \leq \frac12$. What about the range $\frac12 < \delta \leq 1$?

\section*{Acknowledgements} We thank Enrique Trevi\~no for suggestions that improved the readability of the paper.

\providecommand{\bysame}{\leavevmode\hbox to3em{\hrulefill}\thinspace}
\providecommand{\MR}{\relax\ifhmode\unskip\space\fi MR }
\providecommand{\MRhref}[2]{%
  \href{http://www.ams.org/mathscinet-getitem?mr=#1}{#2}
}
\providecommand{\href}[2]{#2}

\end{document}